\documentclass[12pt]{article}
\usepackage{amssymb}
\usepackage{amsfonts}
\usepackage{times}
\usepackage{mathptmx}
\usepackage{amsmath}
\usepackage[usenames]{color}
\usepackage{mathrsfs}
\usepackage{amsfonts}
\usepackage{amssymb,amsmath}
\usepackage{CJK}
\usepackage{cite}
\usepackage{cases}
\usepackage{amsthm}
\def\cl{\centerline}

\def\vs{\vspace*}

\def\Z{\mathbb{Z}}

\def\C{\mathbb{C}}

\numberwithin{equation}{section}
\newtheorem{theo}{Theorem}[section]
\newtheorem{defi}[theo]{Definition}

\newtheorem{rem}[theo]{Remark}
\newtheorem{exam}[theo]{Example}
\newtheorem{lemm}[theo]{Lemma}

\newtheorem{clai}{Claim}

\begin{document}
\begin{center}
{\bf\large Linear commuting maps and skew-symmertric biderivations of the deformative Schr$\ddot{\rm o}$dinger-Virasoro Lie algebras}
\footnote {Supported by the National Natural Science Foundation of China (No. 11431010, 11371278).

$^{\,\S}$Corresponding author: Kun Xu.}
\end{center}

\cl{Guangzhe Fan$^{\,*,\,\dag}$, Yucai Su$^{\,*,\,\P}$, Kun Xu$^{\,*,\,\S}$}

\cl{\small $^{\,\dag}$mathgzfan@126.com}
\cl{\small $^{\,\P}$ycsu@tongji.edu.cn}
\cl{\small $^{\,\S}$1531948@tongji.edu.cn}
\cl{\small $^{\,*}$School of Mathematical Sciences, Tongji University, Shanghai 200092, P. R. China}

\vs{8pt}

{\small\footnotesize
\parskip .005 truein
\baselineskip 3pt \lineskip 3pt
\noindent{{\bf Abstract:} In this paper, we investigate the skew-symmertric biderivations of the deformative Schr$\ddot{\rm o}$dinger-Virasoro Lie algebras
which contain the twisted and original deformative Schr$\ddot{\rm o}$dinger-Virasoro Lie algebras. As an application, we give the explicit form of each linear commuting map on the deformative Schr$\ddot{\rm o}$dinger-Virasoro Lie algebras. In particular, we obtain that there exist non-inner biderivations and non-standard linear commuting maps for the certain deformative Schr$\ddot{\rm o}$dinger-Virasoro Lie algebras.

\vs{5pt}


}}
\parskip .001 truein\baselineskip 6pt \lineskip 6pt

\section{Introduction}
Throughout this paper, we denote by $\Z,\, \C$ the sets of integers, complex numbers respectively. We assume that all vector spaces are based on $\C$, unless otherwise stated.\par

It is well known that the infinite-dimensional Schr$\ddot{\rm o}$dinger Lie algebras and the Virasoro algebra play important roles in many areas of mathematics and physics. In order to investigate the free Schr$\ddot{\rm o}$dinger equations, the twisted and original Schr$\ddot{\rm o}$dinger-Virasoro Lie algebras were introduced by \cite{MH1} in the context of non-equilibrium statistical physics. In \cite{CJ}, the author introduced the deformations of the Schr$\ddot{\rm o}$dinger-Virasoro Lie algebras. In this paper, we consider the following Lie algebras, which are referred to as \emph{deformative Schr$\ddot{\rm o}$dinger-Virasoro Lie algebras ${\mathscr{L}}(\lambda,\mu,s)$}. Fix
 the complex numbers $\lambda, \mu$ and $s=0$ or $\frac12$, the Lie algebra ${\mathscr{L}}(\lambda,\mu,s)$ has $\C$-basis
$\{L_{n}, M_{n}, Y_{n+s}~|~n\in \Z\}$
with the following nontrivial Lie brackets:
\begin{equation}\label{101}
[L_{n}, L_{m}]=(m-n)L_{n+m},
\end{equation}
\begin{equation}\label{102}
[L_{n}, M_{m}]=(m-\lambda n+2\mu)M_{n+m},
\end{equation}
\begin{equation}\label{103}
[L_{n}, Y_{m+s}]=(m+s-\frac{\lambda+1}{2}n+\mu)Y_{n+m+s},
\end{equation}
\begin{equation}\label{104}
[Y_{n+s}, Y_{m+s}]=(m-n)M_{n+m+2s},
\end{equation}
where $n, m \in \Z$. Obviously, ${\mathscr{L}}(\lambda,\mu,s)$ is a generalization of the Schr$\ddot{\rm o}$dinger-Virasoro Lie algebras.
Note that ${\mathscr{L}}(\lambda,\mu,s)$ contains the Schr$\ddot{\rm o}$dinger-Virasoro Lie algebras and their deformations.
For instance,
\begin{itemize}
  \item ${\mathscr{L}}(0,0,s)$ is the well-known \emph{Schr$\ddot{\rm o}$dinger-Virasoro Lie algebra} introduced in \cite{MH1}. Its structures and representations have been widely studied in  \cite{HS, JS2, WY}.

  \item ${\mathscr{L}}(\lambda,\mu,0)$ is called the \emph{twisted deformative Schr$\ddot{\rm o}$dinger-Virasoro Lie algebras} whose structure theories studied in \cite{JS4, WLX}.

  \item ${\mathscr{L}}(\lambda,\mu,\frac{1}{2})$ is called the \emph{original deformative Schr$\ddot{\rm o}$dinger-Virasoro Lie algebras}. In
  \cite{JW,JS3}, the authors obtained some results about structures.

  \item ${\mathscr{L}}(\lambda,0,s)$ is also a class of the deformative Schr$\ddot{\rm o}$dinger-Virasoro Lie algebras. In \cite{FLZ}, the author investigated the Lie bialgebra structures.
\end{itemize}

As is well known, derivations and generalized derivations are very important subjects in the research of both algebras and their generalizations. In recent years, biderivations have aroused many scholars' great interests in \cite{XWH, B1, B2, WYC, WY, FD, ZFLW, N, HWX, DW, C}. In \cite{B2}, Bre$\check{s}$ar showed that all biderivations on commutative prime rings are inner bidrivations and they also determined the biderivations of semiprime rings. The notation of biderivation of Lie algebras was introduced in \cite{WYC}. In addition, in \cite{WY,C} the authors obtained that the skew-symmetric biderivations of the Schr$\ddot{\rm o}$dinger-Virasoro algebra and a simple generalized Witt algebra are inner biderivations. Furthermore, in \cite{HWX} the authors determined all the skew-symmetric biderivations of $W(a, b)$ and found that there exist non-inner biderivations. In 1957, the first important result on linear (or additive) commuting maps was introduced by Posners in \cite{PEC}. The author of \cite{B2} described that commuting maps on an associative algebra have significant application to other important problems (e.g., Lie derivations, biderivations, etc).
Moreover, linear commuting maps of some Lie algebras were investigated in \cite{WY, HWX, DW, C}. Recently, the theory of biderivations and linear commuting maps have become one of research focuses in the Lie theory. Motivated by this reason, we attempt to investigate the linear commuting maps and skew-symmetric biderivations of some important Lie algebras. We know that the linear commuting maps and skew-symmetric biderivations of ${\mathscr{L}}(0,0,0)$ were studied in \cite{WY}. Thus, the results of this paper are more generic.

This paper is organized as follows. In Section $2$, we review some definitions and conclusions on biderivations of Lie algebras. In Section $3$, we compute the skew-symmetric biderivations of ${\mathscr{L}}(\lambda,\mu,s)$. In particular, we obtain that there exist non-inner biderivations for the certain ${\mathscr{L}}(\lambda,\mu,s)$. In Section $4$, we give the explicit form of each linear commuting map on ${\mathscr{L}}(\lambda,\mu,s)$.

\section{Preliminaries and main results}
Firstly, we shall recall some definitions and conclusions about biderivations of Lie algebras in \cite{WY,C}. In this section, we assume that $L$ is a Lie algebra over $\C$.
\begin{defi}\label{2.1}\rm
A bilinear map ${\phi}$: $L\times L\longrightarrow L$ is called \emph{skew-symmetric} if ${\phi}(x, y)$= $-{\phi}(y, x)$ for all $x, y\in$ $L$.
\end{defi}

\begin{defi}\label{2.2}\rm
A bilinear map ${\phi}$: $L\times L\longrightarrow L$ ia called
a \emph{biderivation} if it satisfies the following two axioms:
\begin{center}
$\phi([x,y],z)=[x,\phi(y,z)]+[\phi(x,z),y]$
\end{center}
\begin{center}
$\phi(x,[y,z])=[\phi(x,y),z]+[y,\phi(x,z)]$
\end{center}
for any $x, y, z \in L$.
\end{defi}

\begin{exam}\label{2.3}\rm
Let $\lambda \in$ $\C$, then the map ${\phi_{\lambda}}$: $L\times L\longrightarrow L$, sending ${(x, y)}$ to
$\lambda[x, y]$, is a biderivation of $L$. All biderivations of this kind are called \emph{inner biderivations} of $L$.
\end{exam}

\begin{rem}\label{2.4}\rm
It is straight to check that every inner biderivation ${\phi}_{\lambda}$ is a skew-symmetric biderivation.
\end{rem}

By \cite{WY,C}, the following two results are straightforward to verify.
\begin{lemm}\label{2.5}
Let $\phi$ be a skew-symmetric biderivation on $L$, then
\begin{equation*}
\aligned
&[\phi(x,y),[u,v]]=[[x,y],\phi(u,v)]
\endaligned
\end{equation*}
for any $x,y,u,v\in L$. In particular,
$[\phi(x,y), [x, y]]=0$.

\end{lemm}

\begin{lemm}\label{2.6}\rm
Let $\phi$ be a skew-symmetric biderivation on $L$. If $[x,y]=0$, then $\phi(x,y)\in C_{L}([L,L])$, where $C_{L}([L, L])$ is the centralizer of $[L, L]$.
\end{lemm}

\section{Skew-symmertric biderivations of ${\mathscr{L}}$}
In this section, we would like to compute the skew-symmetric biderivations of ${\mathscr{L}}(\lambda,\mu,s)$, which is denoted by ${\mathscr{L}}$ for simplicity.

\begin{lemm}\label{3.1}\rm
\begin{itemize}
\item[{\rm(1)}] The center $Z({\mathscr{L}})$ of ${\mathscr{L}}$ is given by
\begin{equation*}
Z({\mathscr{L}})=\left\{\begin{array}{lll}
{\C}M_{-2\mu}&\mbox{if \ }{\lambda=0,~ \mu \in \frac{1}{2}}\Z;\\[4pt]
0 & \mbox otherwise.
\end{array}\right.
\end{equation*}
\item[{\rm(2)}] For $x\in {\mathscr{L}}$, denote $\overline{x}=x+[{\mathscr{L}}, {\mathscr{L}}]$. Then
\begin{equation*}
{\mathscr{L}}/[{\mathscr{L}}, {\mathscr{L}}]=\left\{\begin{array}{lll}
{\C}\overline{ Y}_{-\mu}&\mbox{if \ }{\lambda=-3,~ \mu \in s+\Z};\\[4pt]
0 & \mbox otherwise.
\end{array}\right.
\end{equation*}
\item[{\rm(3)}] The centralizer of $[{\mathscr{L}}, {\mathscr{L}}]$ coincides with $Z({\mathscr{L}})$, i.e., $C_{{\mathscr{L}}}([{\mathscr{L}}, {\mathscr{L}}])=Z({\mathscr{L}})$.
\end{itemize}
\end{lemm}
\begin{proof}
It can be easily obtained by the Lie brackets of ${\mathscr{L}}$.
\end{proof}

For convenience, we first introduce two kinds of skew-symmetric biderivations.
 \begin{itemize}
  \item[{\rm(1)}]  For ${\mathscr{L}}$ with ${\lambda=1, \mu  \in  s+\frac{1}{2}}\Z$, define the following skew-symmetric bilinear map: \par
\begin{equation}\label{301}
\phi_0: {\mathscr{L}}\times {\mathscr{L}}\rightarrow{\mathscr{L}}, \quad (L_{n}, L_{m})\mapsto(m-n)M_{n+m-2\mu},
\end{equation}
the others map to 0.
  \item[{\rm(2)}] For ${\mathscr{L}}$ with ${\lambda=1, \mu  \in  s}+\Z$, define the following skew-symmetric bilinear map:
\begin{equation}\label{302}
\begin{array}{ll}
\phi_1:& {\mathscr{L}}\times {\mathscr{L}}\rightarrow{\mathscr{L}}
\\
&(L_{n}, L_{m})\mapsto (m-n)Y_{n+m-\mu},
\\
&(L_{n}, Y_{m+s})\mapsto (m+s-n+\mu)M_{n+m+s-\mu},
\\
&(Y_{m+s}, L_{n})\mapsto (n-m-s-\mu)M_{n+m+s-\mu},
\end{array}
\end{equation}
the others map to $0$.
\end{itemize}
Obviously, we see that $\phi_0$ and $\phi_1$ are skew-symmetric non-inner biderivations of ${\mathscr{L}}$.

\begin{theo}\label{3.2}
Let $\phi$ be a skew-symmetric biderivation of ${\mathscr{L}}$, then we have
\begin{equation*}
\phi(x, y)=\left\{\begin{array}{lll}
\alpha[x, y]+ \beta\phi_0(x, y)&\mbox{if \ }{\lambda=1,~ \mu \in s+ \frac{1}{2}}+\Z;\\[4pt]
\alpha[x, y]+ \beta\phi_0(x, y)+ \gamma\phi_1(x, y) &\mbox{if \ }{\lambda=1,~ \mu \in s}+\Z;\\[4pt]
\alpha[x, y] & \mbox otherwise,
\end{array}\right.
\end{equation*}
for all $x, y \in {\mathscr{L}}$, where $\alpha, \beta, \gamma\in \C$, $\phi_0$ and $\phi_1$ are given by \eqref{301} and \eqref{302}.
\end{theo}
\begin{proof}
We shall complete the proof by verifying the following ten claims.
\begin{clai}\label{3.3}
There exist $\alpha, \beta, \gamma \in \C$ such that
\begin{equation*}
\phi(L_n, L_m)\equiv\left\{\begin{array}{lll}
\alpha[L_n, L_m]+ c_{-\mu}Y_{-\mu} ({\rm mod\,} Z({\mathscr{L}}))&\mbox{if \ }{\lambda=-1,~ \mu \in }s+\Z;\\[4pt]
\alpha[L_n, L_m]+ \beta(m-n)M_{n+m-2\mu} ({\rm mod\,}Z({\mathscr{L}}))&\mbox{if \ }{\lambda=1,~ \mu \in s+\frac{1}{2}}+\Z;\\[4pt]
\alpha[L_n, L_m]+ \beta(m-n)M_{n+m-2\mu}\\[2pt]\ \ +\gamma(m-n)Y_{n+m-\mu} ({\rm mod\,}Z({\mathscr{L}}))&\mbox{if \ }{\lambda=1,~ \mu \in }s+\Z;\\[4pt]
\alpha[L_n, L_m] ({\rm mod\,}Z({\mathscr{L}})) & \mbox otherwise,
\end{array}\right.
\end{equation*}
for all $n, m\in \Z$, where $c_{-\mu}\in \C$.
\end{clai}

If $n=m$, then the fact $[L_n, L_m]=0$ gives that $\phi(L_n, L_m)\in Z({\mathscr{L}})$ by Lemmas \ref{2.6} , \ref{3.1}(3).

For $n\neq m$, suppose that
$\phi(L_{n},L_{m})=\sum_{i \in \Z}(a_{i}L_{i}+b_{i}M_{i}+c_{i+s}Y_{i+s})$ for some $a_{i}, b_{i}, c_{i+s}\in \C$.

According to Lemma \ref{2.5}, one has
\begin{equation*}
\begin{array}{ll}
0&=\frac{1}{m-n}[[L_{n},L_{m}], \phi(L_n, L_m)]
\\[2pt]
&=\Big[L_{n+m},\mbox{$ \sum\limits_{i \in \Z}$}(a_{i}L_{i}+b_{i}M_{i}+c_{i+s}Y_{i+s})\Big]
\\
&=\sum\limits_{i \in \Z}\Big(a_{i}(i-n-m)L_{n+m+i}+b_{i}(i-\lambda(n+m)+2\mu)M_{n+m+i}
\\
& \quad +c_{i+s}(i+s-\frac{\lambda+1}{2}(n+m)+\mu)Y_{n+m+i+s}\Big),
\end{array}
\end{equation*}
which follows that
\begin{equation}\label{303}
\begin{array}{ll}
a_{i}=0 & {if \ }{i\neq n+m},
\\
b_{i}=0 & {if \ }{i\neq \lambda(n+m)-2\mu},
\\
c_{i+s}=0 & {if \ }{i+s\neq \frac{\lambda+1}{2}(n+m)-\mu}.
\end{array}
\end{equation}
Furthermore, by Lemma \ref{2.5} and the equation \eqref{303}, for any $k\neq 0$, then we get
\begin{equation*}
\begin{array}{ll}
0&=[\phi(L_n, L_m), [L_{k},L_{0}]]-[[L_{n},L_{m}], \phi(L_k, L_0)]
\\
&=[a_{n+m}L_{n+m}+b_{\lambda(n+m)-2\mu}M_{\lambda(n+m)-2\mu}+c_{\frac{\lambda+1}{2}(n+m)-\mu}Y_{\frac{\lambda+1}{2}(n+m)-\mu}, -kL_{k}]
\\
& \quad -[(m-n)L_{n+m}, a_{k}L_{k}+b_{k\lambda -2\mu}M_{k\lambda -2\mu}+c_{\frac{\lambda+1}{2}k-\mu}Y_{\frac{\lambda+1}{2}k-\mu}].
\end{array}
\end{equation*}
Thus,
\begin{eqnarray}
\!\!\!\!\!\!\!\!\!\!\!\!&\!\!\!\!\!\!\!\!\!\!\!\!&
k(n+m-k)a_{n+m}L_{n+m+k}=(m-n)(k-n-m)a_{k}L_{n+m+k},\label{304}
\\
\!\!\!\!\!\!\!\!\!\!\!\!&\!\!\!\!\!\!\!\!\!\!\!\!&
k\lambda(n+m-k)b_{\lambda(n+m)-2\mu}M_{\lambda(n+m)-2\mu+k}=\lambda(m-n)(k-n-m)b_{k\lambda-2\mu}M_{n+m+k\lambda-2\mu},\label{305}
\\
\!\!\!\!\!\!\!\!\!\!\!\!&\!\!\!\!\!\!\!\!\!\!\!\!&
\frac{\lambda+1}{2}k(n+m-k)c_{\frac{\lambda+1}{2}(n+m)-\mu}Y_{\frac{\lambda+1}{2}(n+m)-\mu+k}\nonumber\\
\!\!\!\!\!\!\!\!\!\!\!\!&\!\!\!\!\!\!\!\!\!\!\!\!&
\ \hspace*{20pt}\ =
\frac{\lambda+1}{2}(m-n)(k-n-m)c_{\frac{\lambda+1}{2}k-\mu}Y_{n+m+\frac{\lambda+1}{2}k-\mu}.\label{306}
\end{eqnarray}

Due to the arbitrariness of $k$, we  can conclude $a_{n+m}=\alpha(m-n)$ from \eqref{304} for some  $\alpha\in\C$.

In the following we shall consider the coefficients of $M$ and $Y$ in \eqref{305} and \eqref{306}. It can be divided into four cases by choosing different $\lambda$.\par

\textbf{Case 1.} $\lambda=0$.
One can deduce that $c_{\frac{\lambda+1}{2}(n+m)-\mu}=0$ by the arbitrariness of $k$. If $\mu \in \frac{1}{2}\Z$, then we get $\phi(L_n, L_m)=\alpha[L_n, L_m]+b_{-2\mu}M_{-2\mu}$; otherwise, $\phi(L_n, L_m)=\alpha[L_n, L_m]$. \par

\textbf{Case 2.} $\lambda=-1$.
It is easy to get that $b_{\lambda(n+m)-2\mu}=0$. If $\mu \notin s+\Z$, then $c_{-\mu}=0$. Furthermore, we obtain that $\phi(L_n, L_m)=\alpha[L_n, L_m]+c_{-\mu}Y_{-\mu}$ if $\mu \in s+\Z$. \par

\textbf{Case 3.} $\lambda=1$.
If $\mu \in s+\Z$, then we have $\phi(L_n, L_m)=\alpha[L_n, L_m]+b_{n+m-2\mu}M_{n+m-2\mu}+c_{n+m-\mu}Y_{n+m-\mu}$. Moreover, if $\mu \in s+\frac{1}{2}+\Z$, one has $\phi(L_n, L_m)=\alpha[L_n, L_m]+b_{n+m-2\mu}M_{n+m-2\mu}$. Comparing the coefficients in \eqref{305} and \eqref{306}, we conclude $b_{n+m-2\mu}=\frac{b_{k-2\mu}}{k}(n-m)$ and $c_{n+m-\mu}=\frac{c_{k-\mu}}{k}(n-m)$ by the arbitrariness of $k$.
Hence, we could assume $b_{n+m-2\mu}=\beta(m-n)$ and $c_{n+m-\mu}=\gamma(m-n)$ for some $\beta,\gamma\in\C$.
We have obtained the desired result.\par

\textbf{Case 4.} $\lambda\neq 0, \pm1$.
We also obtain that $\phi(L_n, L_m)=\alpha[L_n, L_m]$.

\begin{clai}\label{3.4}
\begin{equation*}
\phi(L_n, M_m)\equiv\left\{\begin{array}{lll}
\alpha[L_n, M_m]+ f_{-\mu}Y_{-\mu} ({\rm mod\,}Z({\mathscr{L}}))&\mbox{if \ }{\lambda=-1,\quad \mu \in s+\Z};\\[4pt]
\alpha[L_n, M_m] ({\rm mod\,}Z({\mathscr{L}})) & \mbox otherwise,
\end{array}\right.
\end{equation*}
for all $n, m\in \Z$, where $f_{-\mu}\in \C$.
\end{clai}
For any fixed $n, m\!\in\! \Z$, suppose
$\phi(L_{n},M_{m})\!=\!\sum_{i \in \Z}(d_{i}L_{i}+e_{i}M_{i}+f_{i+s}Y_{i+s})$ for some $d_{i}, e_{i}, f_{i+s}\!\in\! \C$.
By Lemma \ref{2.5} and Claim \ref{3.3}, for any $k\neq 0$, we have
\begin{equation*}
\begin{array}{ll}
0&=[\phi(L_n, M_m), [L_{0},L_{k}]]-[[L_{n},M_{m}], \phi(L_0, L_k)]
\\[2pt]
&=\Big[\sum\limits_{i \in \Z}(d_{i}L_{i}+e_{i}M_{i}+f_{i+s}Y_{i+s}), kL_{k}\Big]-[(m-\lambda n+2\mu)M_{n+m}, k\alpha L_{k}]
\\
&=\sum\limits_{i \in \Z}\Big(k(k-i)d_{i}L_{i+k}-k(i-k\lambda+2\mu)e_{i}M_{i+k}-k(i+s-\frac{\lambda+1}{2}k+\mu)f_{i+s}Y_{i+s+k}\Big)
\\
& \quad +k\alpha(m-\lambda n+2\mu)(n+m-k\lambda+2\mu)M_{n+m+k},
\end{array}
\end{equation*}
which implies that
\begin{eqnarray}
\label{307}
&&
d_{i}=0  \mbox{ \ if \ }{i\neq k},
\\
\label{308}
&&
(i-k\lambda+2\mu)e_{i}=0  \mbox{ \ if \ }{i\neq n+m},
\\
\label{309}
&&
(n+m-k\lambda+2\mu)(e_{i}-\alpha(m-\lambda n+2\mu))=0  \mbox{ \ if \ }{i=n+m},
\\
\label{310}
&&
f_{i+s}=0  \mbox{ \ if \ }{i+s\neq \frac{\lambda+1}{2}k-\mu}.
\end{eqnarray}
Firstly, one can easily get that $d_{i}=0$ by the arbitrariness of $k$ in \eqref{307}. If $\lambda \neq 0$, one shows that $e_{i}=0$ if $i\neq n+m$ and $e_{n+m}=\alpha(m-\lambda n+2\mu)$ by \eqref{308} and \eqref{309}. If $\lambda \neq -1$, we obtain that $f_{i+s}=0$ from \eqref{310}. Thus,\par
\textbf{Case 1.} $\lambda =-1$.
We have $\phi(L_{n},M_{m})=\alpha[L_{n},M_{m}]+f_{-\mu}Y_{-\mu}$. If $\mu \notin s+\Z$, then $f_{-\mu}=0$.\par

\textbf{Case 2.} $\lambda =0$.
According to \eqref{308} and \eqref{309}, we have $\phi(L_{n},M_{m})=e_{-2\mu}M_{-2\mu}$ if $n+m=-2\mu$. Furthermore,
$\phi(L_{n},M_{m})=\alpha[L_{n},M_{m}]+e_{-2\mu}M_{-2\mu}$ if $n+m \neq -2\mu$. Thanks to Lemma \ref{3.1}(1), this claim holds.\par

\textbf{Case 3.} $\lambda \neq 0, -1$.
One shows that $\phi(L_{n},M_{m})=\alpha(m-\lambda n+2\mu)M_{n+m}=\alpha[L_{n},M_{m}]$. This completes the claim.

\begin{clai}\label{3.5}
\begin{equation*}
\phi(L_n, Y_{m+s})\equiv\left\{\begin{array}{lll}
l_{-\mu}Y_{-\mu} ({\rm mod\,}Z({\mathscr{L}})) \mbox{~when}~ n+m+s= -\mu &\mbox{if \ }{\lambda=-1,~ \mu \in s+\Z};\\[4pt]
\alpha[L_n, Y_{m+s}]\!+\! l_{-\mu}Y_{-\mu} ({\rm mod\,}Z({\mathscr{L}})) \mbox{~when}~ n\!+\!m\!+\!s\!\neq\! -\mu &\mbox{if \ }{\lambda=-1,~ \mu \in s+\Z};\\[4pt]
\alpha[L_n, Y_{m+s}]+ \gamma (m+s-n+\mu)M_{n+m+s-\mu} ({\rm mod\,}Z({\mathscr{L}}))  &\mbox{if \ }{\lambda=1,~ \mu \in s+\Z};\\[4pt]
\alpha[L_n, Y_{m+s}] ({\rm mod\,}Z({\mathscr{L}})) & \mbox{otherwise},
\end{array}\right.
\end{equation*}
for $n, m\in \Z$, where $l_{-\mu}\in \C$.
\end{clai}

For any fixed $n, m\in \Z$, suppose
$\phi(L_{n},Y_{m+s})=\sum_{i \in \Z}(g_{i}L_{i}+h_{i}M_{i}+l_{i+s}Y_{i+s})$ for some $g_{i}, h_{i}, l_{i+s}\in \C$.\par

\textbf{Case 1.} $\lambda=-1, \mu \in s+\Z$.
According to Lemma \ref{2.5} and Claim \ref{3.3}, for any $k\neq 0$, then we get
\begin{equation*}
\begin{array}{ll}
0&=[\phi(L_n, Y_{m+s}), [L_{0},L_{k}]]-[[L_{n},Y_{m+s}], \phi(L_0, L_k)]
\\[2pt]
&=\Big[\sum\limits_{i \in \Z}(g_{i}L_{i}+h_{i}M_{i}+l_{i+s}Y_{i+s}), kL_{k}\Big]-[(m+s+\mu)Y_{n+m+s}, k\alpha L_{k}+c_{-\mu}Y_{-\mu}]
\\
&=\sum\limits_{i \in \Z}\Big(k(k-i)g_{i}L_{i+k}-k(i+k+2\mu)h_{i}M_{i+k}-k(i+s+\mu)l_{i+s}Y_{i+s+k}\Big)
\\
& \quad +c_{-\mu}(m+s+\mu)(n+m+s+\mu)M_{n+m+s-\mu}+k\alpha(m+s+\mu)(n+m+s+\mu)Y_{n+m+s+k}.
\end{array}
\end{equation*}
Firstly, by the arbitrariness of $k, n$ and $m$, one can easily get from the above formula that
\begin{eqnarray}
\label{311}
&&c_{-\mu}=g_{i}=h_{i}=0,
\\
&&
\label{312}
(i+s+\mu)l_{i+s}=0 \mbox{ \ if \ }{i\neq n+m},
\\
\label{313}
&&(n+m+s+\mu)(l_{i+s}-\alpha(m+s+\mu))=0 \mbox{ \ if \ }{i=n+m}.
\end{eqnarray}
By \eqref{312} and \eqref{313}, we conclude that $\phi(L_{n},Y_{m+s})=l_{-\mu}Y_{-\mu}$ if $n+m+s=\mu$, and
$\phi(L_{n},Y_{m+s})=\alpha[L_{n},Y_{m+s}]+l_{-\mu}Y_{-\mu}$ if $n+m+s \neq -\mu$.

\par
\textbf{Case 2.} $\lambda=1, \mu \in s+\Z$.
Using the similar method of Case \ref{3.3}, we conclude the following equalities:
\begin{eqnarray}
\label{314}
&&
g_{i}=0  \mbox{ \ if \ }{i\neq k},
\\
\label{315}
&&
(i-k+2\mu)h_{i}=0 \mbox{ \ if \ }{i\neq n+m+s-\mu},
\\
\label{316}
&&
(n+m+s-k+\mu)(h_{i}-\gamma(m-n+s+\mu))=0 \mbox{ \ if \ }{i=n+m+s-\mu},
\\
\label{317}
&&
(i+s-k+\mu)l_{i+s}=0 \mbox{ \ if \ }{i\neq n+m},
\\
\label{318}
&&
(n+m+s-k+\mu)(l_{i+s}-\alpha(m-n+s+\mu))=0 \mbox{ \ if \ }{i=n+m}.
\end{eqnarray}
Obviously, it can be easily obtained that $g_{i}=0$ by \eqref{314}.
According to \eqref{315}-\eqref{318}, we have $\phi(L_{n},Y_{m+s})=\alpha[L_{n},Y_{m+s}]+\gamma (m-n+s+\mu)M_{n+m+s-\mu}$.

\par
\textbf{Case 3.} Otherwise.
By Lemma \ref{2.5} and Claim \ref{3.3}, for any $k\neq 0$, we have
\begin{equation*}
\begin{array}{ll}
0&=[\phi(L_n, Y_{m+s}), [L_{0},L_{k}]]-[[L_{n},Y_{m+s}], \phi(L_0, L_k)]
\\[2pt]
&=\Big[\sum\limits_{i \in \Z}(g_{i}L_{i}+h_{i}M_{i}+l_{i+s}Y_{i+s}), kL_{k}\Big]-[(m-\frac{\lambda+1}{2}n+s+\mu)Y_{n+m+s}, k\alpha L_{k}]
\\
&=\sum\limits_{i \in \Z}\Big(k(k-i)g_{i}L_{i+k}-k(i-\lambda k+2\mu)h_{i}M_{i+k}-k(i+s-\frac{\lambda+1}{2}k+\mu)l_{i+s}Y_{i+s+k}\Big)
\\
& \quad +k\alpha(m-\frac{\lambda+1}{2}n+s+\mu)(n+m+s-\frac{\lambda+1}{2}k+\mu)Y_{n+m+s+k},
\end{array}
\end{equation*}
which implies that
\begin{eqnarray}
\label{319}
&&
g_{i}=0 \mbox{ \ if \ }{i\neq k},
\\
\label{320}
&&
(i-k\lambda+2\mu)h_{i}=0 \mbox{ \ if \ }{i\neq k\lambda-2\mu},
\\
\label{321}
&&
(i+s-\frac{\lambda+1}{2}k+\mu)l_{i+s}=0 \mbox{ \ if \ }{i\neq n+m},
\\
\label{322}
&&
(n+m+s-\frac{\lambda+1}{2}k+\mu)(l_{i+s}-\alpha(m-\frac{\lambda+1}{2}n+s+\mu))=0 \mbox{ \ if \ }{i=n+m}.
\end{eqnarray}
We get that $g_{i}=0$ by the arbitrariness of $k$ in \eqref{319}. If $\lambda=0$, we obtain that $\phi(L_{n},Y_{m+s})=\alpha[L_{n},Y_{m+s}]+h_{-2\mu}M_{-2\mu}$.
If $\lambda \neq 0$, by the arbitrariness of $k$ in \eqref{320}, it follows $h_{i}=0$. Since $\lambda \neq -1$ or $\mu \notin s+\Z$, we get
$l_{i+s}=0$ if $i\neq n+m$ by \eqref{321} and \eqref{322}. Thus, $\phi(L_{n},Y_{m+s})=\alpha[L_{n},Y_{m+s}]$.

\begin{clai}\label{3.6}
\begin{equation*}
\phi(Y_{n+s}, Y_{m+s})\equiv\left\{\begin{array}{lll}
\alpha[Y_{n+s}, Y_{m+s}]+ r_{-\mu}Y_{-\mu} ({\rm mod\,}Z({\mathscr{L}}))&\mbox{if \ }{\lambda=-1,~ \mu \in s+\Z};\\[4pt]
\alpha[Y_{n+s}, Y_{m+s}] ({\rm mod\,}Z({\mathscr{L}})) & \mbox otherwise,
\end{array}\right.
\end{equation*}
for all $n, m \in \Z$, where $r_{-\mu} \in \C$.
\end{clai}

For any fixed $n, m\in \Z$, assume
$\phi(Y_{n+s}, Y_{m+s})=\sum_{i \in \Z}(p_{i}L_{i}+q_{i}M_{i}+r_{i+s}Y_{i+s})$ for some $p_{i},q_{i},r_{i+s}\in \C$.
By Lemma \ref{2.5} and Claim \ref{3.3}, for any $k\neq 0$, we have
\begin{equation*}
\begin{array}{ll}
0&=[\phi(Y_{n+s}, Y_{m+s}), [L_{0},L_{k}]]-[[Y_{n+s}, Y_{m+s}], \phi(L_0, L_k)]
\\[2pt]
&=\Big[\sum\limits_{i \in \Z}(p_{i}L_{i}+q_{i}M_{i}+r_{i+s}Y_{i+s}), kL_{k}\Big]-[(m-n)M_{n+m+2s}, k\alpha L_{k}]
\\
&=\sum\limits_{i \in \Z}\Big(k(k-i)p_{i}L_{i+k}-k(i-k\lambda+2\mu)q_{i}M_{i+k}-k(i+s-\frac{\lambda+1}{2}k+\mu)r_{i+s}Y_{i+s+k}\Big)
\\
& \quad +k\alpha(m-n)(n+m+s-k\lambda+2\mu)M_{n+m+2s+k},
\end{array}
\end{equation*}
which implies that
\begin{eqnarray}
\label{323}
&&
p_{i}=0 \mbox{ \ if \ }{i\neq k},
\\
\label{324}
&&
(i-k\lambda+2\mu)q_{i}=0 \mbox{ \ if \ }{i\neq n+m+2s},
\\
\label{325}
&&
(n+m+2s-k\lambda+2\mu)(q_{i}-\alpha(m-n))=0 \mbox{ \ if \ }{i=n+m+2s},
\\
\label{326}
&&
r_{i+s}=0 \mbox{ \ if \ }{i+s\neq \frac{\lambda+1}{2}k-\mu}.
\end{eqnarray}
Hence, $p_{i}=0$ by \eqref{323}. If $\lambda \neq 0$, then we have $q_{i}=0$ if $i\neq n+m+2s$ and $q_{n+m+2s}=\alpha(m-n)$ from \eqref{324} and \eqref{325}). If $\lambda \neq -1$, we obtain that $r_{i+s}=0$ by the arbitrariness of $k$ in \eqref{326}. Thus,\par

\textbf{Case 1.}  $\lambda =-1$.
It has $\phi(Y_{n+s}, Y_{m+s})=\alpha[Y_{n+s}, Y_{m+s}]+r_{-\mu}Y_{-\mu}$. If $\mu \notin s+\Z$, then $r_{-\mu}=0$.\par

\textbf{Case 2.} $\lambda =0$.
We have $\phi(Y_{n+s}, Y_{m+s})=q_{-2\mu}M_{-2\mu}$ if $n+m+2s=-2\mu$, and
$\phi(Y_{n+s}, Y_{m+s})=\alpha[Y_{n+s}, Y_{m+s}]+q_{-2\mu}M_{-2\mu}$ if $n+m+2s \neq -2\mu$.\par

\textbf{Case 3.} $\lambda \neq 0, -1$.
Thus, $\phi(Y_{n+s}, Y_{m+s})=\alpha[Y_{n+s}, Y_{m+s}]$.
This claim holds. \par

Now we would like to simplify the forms of Claims \ref{3.3} to \ref{3.6}.  \par
 If $\lambda=-1$ and $\mu \in s+\Z$. On one hand, for any $n, m \in \Z$, since $\phi$ is a biderivation, we have
\begin{equation}\label{327}
\phi(L_{n}, [Y_{s}, Y_{m+s}])=[\phi(L_{n}, Y_{s}), Y_{m+s}]+[Y_{s}, \phi(L_{n}, Y_{m+s})] .
\end{equation}
If $m\neq 0$, $n+s\neq -\mu$ and $n+m+s\neq -\mu$, from Claims \ref{3.4}, \ref{3.5} and \eqref{327}, then we get that
\begin{equation*}
\begin{array}{ll}
& \alpha m(n+m+2s+2\mu)M_{n+m+2s}+m f_{-\mu}Y_{-\mu}
\\
=&\alpha(s+\mu)(m-n)M_{n+m+2s}+(m+s+\mu)l_{-\mu}M_{m+s-\mu}
\\
&+\alpha(n+m)(m+s+\mu)M_{n+m+2s}-(s+\mu)l_{-\mu}M_{s-\mu},
\end{array}
\end{equation*}
which infers that
\begin{equation}\label{328}
f_{-\mu}=l_{-\mu}=0.
\end{equation}
If $m\neq 0$ and $n+m+s= -\mu$, by Claims \ref{3.4}, \ref{3.5} and \eqref{327}, it is enough to see that
\begin{equation*}
\alpha m(s+\mu)M_{s-\mu}=\alpha(s+\mu)(m-n)M_{s-\mu}-(s+\mu)l_{-\mu}M_{s-\mu},
\end{equation*}
which implies that $l_{-\mu}=-\alpha n$ if $\mu \neq -s$. Thus, one shows that
\begin{equation}\label{329}
\phi(L_{n}, Y_{m+s})=\left\{\begin{array}{lll}
l_{s}Y_{s} &\mbox{if \ }{n+m=0};\\[4pt]
\alpha[L_{n}, Y_{m+s}] & \mbox {if \ }{n+m\neq0},
\end{array}\right.
\end{equation}
for any $n, m \in \Z$.

On the other hand, for any $n, m \in \Z$, we have
\begin{equation}\label{330}
\phi([L_{0}, Y_{n+s}], Y_{m+s})=[L_{0}, \phi(Y_{n+s}, Y_{m+s})]+[\phi(L_{0}, Y_{m+s}), Y_{n+s}].
\end{equation}
If $m\neq 0$ and $n\neq -s-\mu$, according to Claim \ref{3.6}, \eqref{329} and \eqref{330}, we can easily get that
\begin{equation}\label{331}
r_{-\mu}=0.
\end{equation}
If $m=0$ and $n\neq 0$, by Claim \ref{3.6}, \eqref{329}-\eqref{331}, one shows that
\begin{equation*}
-\alpha n(n+s+\mu)M_{n+2s}=-\alpha n(n+2s+2\mu)M_{n+s}+nl_{s}M_{n+s}.
\end{equation*}
which implies that $l_{s}=\alpha (s+\mu)$. Hence, for $n, m \in \Z$, we have
\begin{equation}\label{332}
\phi(L_{n}, Y_{m+s})=\alpha[L_{n}, Y_{m+s}].
\end{equation}

\begin{clai}\label{3.7}
There exist $\alpha, \beta, \gamma\in \C$ such that
\begin{equation*}
\phi(L_n, L_m)\equiv\left\{\begin{array}{lll}
\alpha[L_n, L_m]+ \beta\phi_0(L_n, L_m) ({\rm mod\,}Z({\mathscr{L}}))&\mbox{if \ }{\lambda=1,~ \mu \in s+\frac{1}{2}}+\Z;\\[4pt]
\alpha[L_n, L_m]+ \beta\phi_0(L_n, L_m)+\gamma\phi_1(L_n, L_m) ({\rm mod\,}Z({\mathscr{L}}))&\mbox{if \ }{\lambda=1,~ \mu \in s+\Z};\\[4pt]
\alpha[L_n, L_m] ({\rm mod\,}Z({\mathscr{L}})) & \mbox otherwise,
\end{array}\right.
\end{equation*}
for all $n, m\in \Z$, where $\phi_0$ and $\phi_1$ are given by \eqref{301} and \eqref{302}.
\end{clai}
It is obvious by Claim \ref{3.3} and \eqref{311}.
\begin{clai}\label{3.8}
$\phi(L_n, M_m)\equiv \alpha[L_n, M_m]({\rm mod\,}Z({\mathscr{L}}))$ for all $n, m \in \Z$.
\end{clai}
It is easily obtained  by Claim \ref{3.4} and \eqref{328}.

\begin{clai}\label{3.9}

\begin{equation*}
\phi(L_n, Y_{m+s})\equiv\left\{\begin{array}{lll}
\alpha[L_n, Y_{m+s}]+ \gamma\phi_1(L_n, Y_{m+s}) ({\rm mod\,}Z({\mathscr{L}}))&\mbox{if \ }{\lambda=1,~ \mu \in s+\Z};\\[4pt]
\alpha[L_n, Y_{m+s}] ({\rm mod\,}Z({\mathscr{L}})) & \mbox otherwise,
\end{array}\right.
\end{equation*}
for all $n, m\in \Z$, where $\phi_1$ is given by \eqref{302}.
\end{clai}
 It can be easily obtained by Claim \ref{3.5} and \eqref{332}.
\begin{clai}\label{3.10}
$\phi(Y_{n+s}, Y_{m+s})\equiv \alpha[Y_{n+s}, Y_{m+s}]({\rm mod\,}Z({\mathscr{L}}))$ for $n, m \in \Z$.
\end{clai}
It is obvious by Claim \ref{3.6} and \eqref{331}.
\begin{clai}\label{3.11}
$\phi(M_{n}, M_{m})\equiv 0({\rm mod\,}Z({\mathscr{L}}))$ for $n, m \in \Z$.
\end{clai}
For any fixed $n, m \in \Z$, since $[M_{n}, M_{m}]=0$, it is enough to see that $\phi(M_{n}, M_{m})\in C_{\mathscr{L}}([{\mathscr{L}}, {\mathscr{L}}])$ by Lemma \ref{2.6}. Thus, this claim holds by Lemma \ref{3.1}(3).

\begin{clai}\label{3.12}
$\phi(M_{n}, Y_{m+s})\equiv 0({\rm mod\,}Z({\mathscr{L}}))$ for $n, m \in \Z$.
\end{clai}
It can be obtained by using the similar method of claim \ref{3.11}.

Now, by Claims \ref{3.7}-\ref{3.12}, for any $x, y \in {\mathscr{L}}$, we have
\begin{equation}\label{333}
\phi(x, y)=\left\{\begin{array}{lll}
\alpha[x, y]+ \beta\phi_0(x, y) ({\rm mod\,}Z({\mathscr{L}}))&\mbox{if \ }{\lambda=1,~ \mu \in s+\frac{1}{2}}+\Z;\\[4pt]
\alpha[x, y]+ \beta\phi_0(x, y)+ \gamma\phi_1(x, y) ({\rm mod\,}Z({\mathscr{L}}))&\mbox{if \ }{\lambda=1,~ \mu \in s+\Z};\\[4pt]
\alpha[x, y] ({\rm mod\,}Z({\mathscr{L}})) & \mbox otherwise.
\end{array}\right.
\end{equation}
If $\lambda \neq 0$ or $\mu \notin \frac{1}{2}\Z$, we get $Z({\mathscr{L}})=0$ by Lemma \ref{3.1}(1). Hence, Theorem \ref{3.2} holds.
Then, consider the case $\lambda=0$, $\mu \in \frac{1}{2}\Z$. By Lemma \ref{3.1}(1) and \eqref{333}, we may assume that
$\phi(x, y)=\alpha[x, y]+ \theta(x, y)M_{-2\mu}$, where $\theta$ is a bilinear function from ${\mathscr{L}} \times {\mathscr{L}}$ to $\C$.
We need to show that $\theta$ is the zero function. In fact, by
\begin{equation*}
\phi([x, y], z)=[x, \phi(y, z)]+[\phi(x, z), y],
\end{equation*}
we have that $\theta([x, y], z)=0$ for all $x, y, z\in {\mathscr{L}}$. Note that ${\mathscr{L}}=[{\mathscr{L}}, {\mathscr{L}}]$ in this case by Lemma \ref{3.1}(2).
It follows that $\theta$ is exactly the zero function, as desired.
\end{proof}

\section{Linear commuting maps on ${\mathscr{L}}$}
Let $\cal{A}$ be an associative algebra, a map $\psi: \cal{A}\rightarrow \cal{A}$ is called a commuting map if $\psi(x)x=x\psi(x)$ for
all $x \in \cal{A}$. If we denote $[x, y]=xy-yx$ for $x, y \in \cal{A}$, then a commuting map $\psi$ on $\cal{A}$
can also be defined as $[\psi(x), x]=0$ for all $x\in \cal{A}$.
 \begin{defi}\label{4.1}\rm
 Let $L$ be a Lie algebra, a map $\psi: L \rightarrow L$ is called \emph{commuting} if
 \begin{equation*}
[\psi(x), x]=0 \quad \forall ~x \in L.
\end{equation*}
\end{defi}
\begin{defi}\label{4.2}\rm
Define the following map
\begin{equation*}
\psi(x)=\alpha x+f(x),  \quad \forall ~x\in {\mathscr{L}}
\end{equation*}
on ${\mathscr{L}}$, where $\alpha \in \C$, $ f :{\mathscr{L}}\rightarrow Z({\mathscr{L}})$. Obviously, $\psi(x)$ is a linear commuting map. We call such a map a \emph{standard linear commuting map} on ${\mathscr{L}}$. Other linear commuting maps are called \emph{non-standard}.
\end{defi}
Now we apply Theorem \ref{3.2} to describe linear commuting maps on the deformative Schr$\ddot{\rm o}$dinger-Virasoro Lie algebras ${\mathscr{L}}$ . Obviously, the identity map is a standard linear commuting map. For ${\mathscr{L}}$ with $\lambda=0, \mu \in \frac{1}{2}\Z$ and some
linear function$f$ from ${\mathscr{L}}$ to $\C$, by Lemma \ref{3.1}(1), the map $x\mapsto f(x)M_{-2\mu}$ is also a standard linear commuting map.

For convenience, we first introduce two kinds of linear commuting maps.
 \begin{itemize}
  \item[{\rm(1)}]  For ${\mathscr{L}}$ with ${\lambda=1, \mu \in s+\frac{1}{2}}\Z$, define the following linear map:
\begin{equation}
\psi_0: {\mathscr{L}}\rightarrow {\mathscr{L}}, \quad L_{n}\mapsto M_{n-2\mu},
\end{equation}
the others map to 0.
  \item[{\rm(2)}] For ${\mathscr{L}}$ with ${\lambda=1, \mu \in s+\Z}$, define the following linear map:
\begin{equation}
\psi_1: {\mathscr{L}}\rightarrow {\mathscr{L}},\quad L_{n}\mapsto Y_{n-\mu}, \quad  Y_{n+s}\mapsto M_{n+s-\mu},
\end{equation}
others map to $0$.
\end{itemize}

It can be easily checked that $\psi_0$ and $\psi_1$ are non-standard linear commuting maps.

\begin{theo}\label{4.3}
Each linear commuting map $\psi$ on ${\mathscr{L}}$ is one of the following forms:
\begin{equation*}
\psi(x)=\left\{\begin{array}{lll}
\alpha(x)+ \beta\psi_0(x)&\mbox{if \ }{\lambda=1,~ \mu \in s+\frac{1}{2}}+\Z;\\[4pt]
\alpha(x)+ \beta\psi_0(x)+ \gamma\psi_1(x) &\mbox{if \ }{\lambda=1,~ \mu \in s+\Z};\\[4pt]
\alpha(x)+ f(x)M_{-2\mu} &\mbox{if \ }{\lambda=0,~ \mu \in \frac{1}{2}}\Z;\\[4pt]
\alpha(x) & \mbox otherwise,
\end{array}\right.
\end{equation*}
for any $x \in {\mathscr{L}}$, where $\alpha, \beta, \gamma\in \C$, $\psi_0$ and $\psi_1$ are given by {\rm(4.1)} and {\rm(4.2)}, and $f$ is a linear function form
${\mathscr{L}}$ to $\C$.
\end{theo}
\begin{proof}
We assume that $\psi$ is a linear commuting map on ${\mathscr{L}}$. Define
\begin{equation}\label{403}
\phi: {\mathscr{L}}\times {\mathscr{L}}\rightarrow {\mathscr{L}} \quad (x, y)\mapsto[\psi(x), y] \quad \forall ~x, y \in {\mathscr{L}}.
\end{equation}
Thus, one has
\begin{equation*}
\phi(x, [y, z])=[\phi(x, y), z]+[y, \phi(x, z)] \quad \forall ~x, y, z \in {\mathscr{L}}.
\end{equation*}
Obviously, $\phi$ is a derivation with respect to the second component. Since $[\psi(x), y]=[x, \psi(y)]$, we conclude that
$\phi$ is also a derivation with respect to the first component. Thus, $\phi$ is a biderivation of ${\mathscr{L}}$. From \eqref{403}, we know that $\phi$ is skew-symmetric. According to Theorem \ref{3.2},  there exist $\alpha, \beta, \gamma \in \C$ such that
\begin{equation*}
\phi(x, y)=\left\{\begin{array}{lll}
\alpha[x, y]+ \beta\phi_0(x, y)&\mbox{if \ }{\lambda=1,~ \mu \in s+\frac{1}{2}}+\Z;\\[4pt]
\alpha[x, y]+ \beta\phi_0(x, y)+ \gamma\phi_1(x, y)&\mbox{if \ }{\lambda=1,~ \mu \in s+\Z};\\[4pt]
\alpha[x, y] & \mbox otherwise,
\end{array}\right.
\end{equation*}
for all $x, y \in {\mathscr{L}}$, where $\phi_0$ and $\phi_1$ are given by \eqref{301} and \eqref{302}. Furthermore, by \eqref{403}, then we obtain that
\begin{equation*}
[\psi(x)-\alpha x, y]=\left\{\begin{array}{lll}
\beta\phi_0(x, y)&\mbox{if \ }{\lambda=1,~ \mu \in s+\frac{1}{2}}+\Z;\\[4pt]
\beta\phi_0(x, y)+\gamma\phi_1(x, y) &\mbox{if \ }{\lambda=1,~ \mu \in s+\Z};\\[4pt]
0 & \mbox otherwise.
\end{array}\right.
\end{equation*}\par
\textbf{Case 1.} $\lambda=1, \mu \in s+\frac{1}{2}+\Z$.
By the definition of $\phi_{0}$ given by \eqref{301},  we conclude that $\psi(x)-\alpha x=\beta \psi_{0}(x)$.
The conclusion holds.\par

\textbf{Case 2.} $\lambda=1, \mu \in s+\Z$.
By the definition of $\phi_{0}$ and $\phi_{1}$ given by \eqref{301} and \eqref{302},  one follows that
$\psi(x)-\alpha x=\beta \psi_{0}(x)+\gamma\psi_1(x)$.\par

\textbf{Case 3.} $\lambda \neq 1$ or $\mu \notin s+\frac{1}{2}\Z$.
Obviously, we have $\psi(x)-\alpha x \in Z({\mathscr{L}})$.
By Lemma \ref{3.1}(1), we only need to consider the case $\lambda=0, \mu \in \frac{1}{2}\Z$. Define $f$ by letting $\psi(x)-\alpha x=f(x)M_{-2\mu}$,
then $f$ is a linear function from ${\mathscr{L}}$ to $\C$, and $\psi(x)=\alpha x+f(x)M_{-2\mu}$. This completes the proof.
\end{proof}

\end{document}